\documentclass{amsart}

\numberwithin{equation}{section}
\usepackage{graphicx}
\usepackage{subfigure}
\usepackage{booktabs}

\newtheorem{theorem}{Theorem}[section]

\DeclareMathSymbol{\Z}{\mathalpha}{AMSb}{'132}
\DeclareMathSymbol{\Real}{\mathalpha}{AMSb}{"52}
\DeclareMathOperator*{\Span}{span}
\DeclareMathOperator{\Realpart}{Re}
\DeclareMathOperator{\Imaginarypart}{Im}

\newcommand{\mua}{\mu_\text{a}}
\newcommand{\mus}{\mu_\text{s}}

\title[Accurate Quadrature Rule on the Sphere for 3D-RTE]{An Accurate Quadrature Rule on the Sphere for Fast Computation of the Radiative Transport Equation}

\author[H.~Fujiwara]{Hiroshi Fujiwara}
\address{Graduate School of Informatics, Kyoto University,
Yoshida-Honmachi, Sakyo-ku, Kyoto 606-8501}
\email{fujiwara@acs.i.kyoto-u.ac.jp}



\begin{document}

\maketitle

\begin{abstract}
We present an accurate quadrature formula on the sphere
with less localized quadrature points
for efficient numerical computation of the radiative transport equation (RTE)
in the three dimensions.
High accuracy of the present method
dramatically reduces computational resources
and fast computation of 3D RTE is achieved.
\end{abstract}

\section{Introduction}
We consider an approximation of quadrature on the unit surface,
\[
Q(f) = \int_{S^2} f(\xi) d\sigma_\xi
\approx Q_K(f) = \sum_{k=1}^K w_k f(\theta_k,\phi_k),
\]
where $f$ is a continuous function on $S^2$,
$d\sigma_\xi$ is the surface element,
and $\xi_k = (\theta_k,\phi_k)$ is the standard polar coordinate on $S^2$.

One of its application is fast numerical computation of
the radiative transport equation (RTE) in the three dimensions
which is a mathematical model of near infrared light propagation
in human bodies~\cite{arridge,yamada2014}.
In brain science,
detecting NIR light absorption by hemoglobin is expected to be
a new modality for non-invasive monitoring of our brain activities.

Let $\Omega$ be a bounded domain with piecewise smooth boundary.
We consider the following boundary value problem of RTE
\begin{subequations}\label{eq:rte}
\begin{alignat}{2}
-\xi\cdot\nabla_x I - (\mus + \mua) I
  + \mus \int_{S^2} p(\xi,\xi') I(x,\xi') d\sigma_{\xi'} &= q,
&\quad& \text{in $\Omega \times S^2$} \label{eq:rte:eq},
\\
I(x,\xi) &= I_1(x,\xi), &\quad& \text{on $\Gamma_-$},\label{eq:boundary}
\end{alignat}
\end{subequations}
where $I = I(x,\xi)$
is light intensity at a position $x\in\Omega\subset\Real^3$
with a direction $\xi \in S^2$.
The function $I_1$ is given on 
$\Gamma_- = \bigl\{(x,\xi) \:;\: x \in \partial\Omega, n(x)\cdot\xi < 0\bigr\}$,
where $n(x)$ is the outer unit normal vector to $\partial\Omega$.
The coefficients $\mua$ and $\mus$ represent absorption and 
scattering respectively, and $p(\xi,\xi')$ is
called a scattering phase function
which represents a conditional probability of a photon
changing its velocity from $\xi'$ to $\xi$ by a collision with a scatterer.

Discretizing $\nabla_x I$ by finite difference
and scattering integral by a numerical quadrature rule,
a system of linear algebraic equations is obtained~\cite{klose}.
Since \eqref{eq:rte} in the spatial three dimensions is essentially
a five dimensional problem, numerical computation requires
huge resources (time and storage).
And for its unique solvability,
errors in the quadrature rule should
be sufficiently small~\cite{fujiwara:jascome2011}.
In some examples, computation of the scattering integral spends
over 80\% of total computational time~\cite{fujiwara:jascome2013}.
This means that a novel treatment of scattering integral
is effective for fast computation.

A simple approximation for $Q(f)$ is repeating the trapezoidal rule
to both zenith and azimuthal directions as
\begin{equation}\label{eq:tt}
Q(f)
= \int_0^{2\pi} \int_0^\pi f(\theta, \phi)\sin\theta\:d\theta\:d\phi
\approx \dfrac{\pi}{M_\theta}\dfrac{2\pi}{M_\phi}
  \sum_{m=0}^{M_\theta} \sum_{n=0}^{M_\phi-1} f(\theta_m, \phi_n) \sin\theta_m,
\end{equation}
where $\theta_m = m\pi/M_\theta$ and $\phi_n = 2n\pi/M_\phi$.
This is convenient in application
since both the weights and the nodes $(\theta_m, \phi_n)$ are
explicitly known. 
However a large number of nodes should be taken due to its low accuracy.
Moreover, they localize near the poles
and associated weights are small relatively.
This indicates that
unknowns near the poles formally introduced in discretization
have less meaning.

\section{A New Quadrature Rule on the Sphere}
We construct high-accurate quadrature rule on the sphere
with less localized quadrature points
to reduce a number of discretization points
for the velocity direction $\xi \in S^2$ in \eqref{eq:rte:eq}.

To state more precisely,
we introduce some basic concepts from \cite{sobolev1997}.
For a linear subspace $V \subset C(S^2)$,
we say $Q_K$ is \emph{exact} on $V$ if $Q_K(f) = Q(f)$ for
any function $f \in V$.
For a finite rotation group $G \subset SO(3)$, $Q_K$ is said to be
\emph{invariant} under $G$ if the set of quadrature points $\{\xi_k\}$
is a disjoint union of $G$-orbits,
$\{\xi_k\} = \{g\xi_1'\:;\:g \in G\} \cup \dotsb \cup \{g\xi_s'\:;\: g \in G\}$,
and $w_k = w_j$ if $\xi_k$ and $\xi_j$ belong to
the same orbit.
For $f \in V$ and $g \in G$, we set $f_g(x) = f(gx)$ and
$V_G = \{f \in V\:;\: \text{$f_g = f$ for any $g \in G$}\}$.

For high-accuracy of $Q_K$, we require that it is exact on
\[
\Pi^N = \Span \{Y_n^m\:;\: |m| \leq n \leq N\}
\]
for some non-negative integer $N$, where
\[
Y_n^m(\theta,\phi)
= (-1)^m\sqrt{\dfrac{2n+1}{4\pi}\dfrac{(n-m)!}{(n+m)!}}
   P_n^m(\cos\theta)e^{im\phi},
\quad
|m| \leq n,
\]
is a spherical harmonic of degree $n$ and order $m$~\cite{arfken}.
The functions $\{Y_n^m\:;\:|m|\leq n\}$ form
a complete orthogonal system of $L^2(S^2)$ with
\[
Q(Y_n^m) = \sqrt{4\pi} \delta_{n0} \delta_{m0},
\]
thus $Q_K$ satisfies
\begin{equation}\label{eq:tobesolved1}
Q_K(Y_n^m) = \sum_{k=1}^K w_k Y_n^m(\theta_k, \phi_k) = \sqrt{4\pi}\delta_{n0}\delta_{m0}, \quad
\text{for any $|m| \leq n \leq N$}.
\end{equation}
Exactness on $\Pi^N$ is motivated by
the spherical harmonic expansion of an analytic function $f$ on $S^2$,
\[
f(\theta,\phi) = \sum_{n=0}^\infty \sum_{|m| \leq n} a_{nm} Y_n^m(\theta,\phi)
\]
where coefficients $|a_{nm}|$ decrease exponentially
with respect to $n$~\cite{sobolev1992}.
Hence the error $|Q_K(f)-Q(f)|$ decays rapidly with respect to $N$.

For less localized distribution of quadrature points,
we require that $Q_K$ is also rotationally invariant under the icosahedral group.
Hence we call the proposed approximation
{\em RIQS20} (Rotationally Invariant Quadrature rule on the Sphere under the icosahedral group).

If $Q_K$ is invariant under some finite rotation group $G$ and
exact on $\Pi^N_G$, then it is exact on $\Pi^N$~\cite{sobolev1962}.
Therefore \eqref{eq:tobesolved1} for all $|m| \leq n \leq N$ are
redundant.
The next theorem gives an example of reduction
of \eqref{eq:tobesolved1} using symmetries of $Y_n^m$.
It also reduces \eqref{eq:tobesolved1} to real-valued equations.

\begin{theorem}\label{thm:sym}
Suppose that $\{w_k, \theta_k, \phi_k\}$ in $Q_K$ satisfies the following conditions:
(i) for any $(\theta_k, \phi_k)$, there uniquely exists $(\theta_j, \phi_j)$
such that $(\theta_j, \phi_j) = (\theta_k, \phi_k+\pi)$ and $w_j = w_k$, and
(ii) for any $(\theta_k, \phi_k)$, there uniquely exists $(\theta_j, \phi_j)$
such that $(\theta_j, \phi_j) = (\theta_k+\pi, \pi-\phi_k)$ and $w_j = w_k$.
Then, the system \eqref{eq:tobesolved1} is equivalent to
\begin{subequations}\label{eq:tobesolved}
\begin{alignat}{3}
\Realpart\sum_{k=1}^K w_k Y_n^m(\theta_k, \phi_k) &= \sqrt{4\pi}\delta_{n0}\delta_{m0}, &\quad&
\text{if $n$ is even or $0$}, \label{eq:tobesolved:even}\\
\Imaginarypart\sum_{k=1}^K w_k Y_n^m(\theta_k, \phi_k) &= 0, &\quad&
\text{if $n$ is odd},\label{eq:tobesolved:odd}
\end{alignat}
\end{subequations}
for any $(m,n) \in \bigl\{(2\mu,2\nu)\:;\: 0 \leq 2\mu \leq 2\nu \leq N, \mu,\nu \in \Z\bigr\}\cup
\bigl\{(2\mu,2\nu+1)\:;\: 0 < 2\mu \leq 2\nu+1 \leq N, \mu,\nu \in \Z\bigr\}$.
\end{theorem}

\begin{proof}
Note that $Y_n^m$ satisfies following symmetries,
\begin{align}
Y_n^m(\theta,\phi+\pi) &= (-1)^m Y_n^m(\theta,\phi), \label{eq:sym1}\\
Y_n^m(\theta+\pi,\pi-\phi) &= (-1)^n \overline{Y_n^m(\theta,\phi)}, \label{eq:sym2}
\intertext{and}
Y_n^{-m}(\theta+\pi,\pi-\phi) &= (-1)^{n+m} Y_n^m(\theta,\phi).\label{eq:sym3}
\end{align}
From the assumption (i) and \eqref{eq:sym1}, we have
\[
\sum_{k=1}^K w_k Y_n^m(\theta_k, \phi_k) = 0, \quad \text{if $m$ is odd}.
\]
Similarly, from the assumption (ii) and \eqref{eq:sym2}, we have
\begin{alignat*}{3}
\Imaginarypart \sum_{k=1}^K w_k Y_n^m(\theta_k,\phi_k) &= 0, &\quad&
\text{if $n$ is even},\\
\Realpart \sum_{k=1}^K w_k Y_n^m(\theta_k,\phi_k) &= 0, &\quad&
\text{if $n$ is odd},
\end{alignat*}
and
\[
\sum_{k=1}^K w_k Y_n^0(\theta_k, \phi_k) = 0, \quad \text{if $n$ is odd}.
\]
To sum up, we obtain Table~\ref{tbl:sumYnm} under conditions (i) and (ii).
\begin{table}[t]
\caption{\label{tbl:sumYnm}Real and imaginary parts of $Q_K(Y_n^m)$ under conditions (i) and (ii)}
\begin{center}
\begin{tabular}{cc|cc}
\toprule
$n$ & $m$ & Real part & Imaginary part \\
\midrule
even, $0$ & even, $0$ & \eqref{eq:tobesolved:even} & $0$ \\
          & odd  & $0$ & $0$ \\
\midrule
odd       & even, $\neq 0$ & $0$ & \eqref{eq:tobesolved:odd} \\
          & odd, $0$      & $0$ & $0$ \\
\bottomrule
\end{tabular}
\end{center}
\end{table}
Finally,
\[
\sum_{k=1}^K w_k Y_n^m(\theta_k, \phi_k) = 0, \quad m > 0
\]
leads
\[
\sum_{k=1}^K w_k Y_n^{-m}(\theta_k, \phi_k) = 0, \quad m > 0
\]
from the assumption (ii) and \eqref{eq:sym3}.
This concludes the proof.
\end{proof}

The system of equations~\eqref{eq:tobesolved} is solved numerically
with the Newton iteration or the homotopy method in this study.
In the computation we use
a representation of the icosahedral group
which does not change the icosahedron with vertices
\begin{equation}\label{eq:icosahedron}
(\pm1,0,\pm\alpha),\quad
(\pm\alpha,\pm1,0),\quad
(0,\pm\alpha,\pm1),\quad
\quad
\alpha = \dfrac{\sqrt{5}-1}{2}.
\end{equation}
It gives an orbit which satisfies the conditions in Theorem~\ref{thm:sym}~\cite{fujiwara:RIQS}.

\section{Numerical Examples}

Constructed quadrature points of RIQS20 with the degree $75$
are shown in Figure~\ref{fig:nodes:riqs20},
where solid curves show the projection of the icosahedron \eqref{eq:icosahedron}.
The number of quadrature points on $S^2$ is $1932$,
and they are less localized by virtue of rotational invariance
than those of \eqref{eq:tt} shown in Figure~\ref{fig:nodes:tt}.
\begin{figure}[h]
\begin{minipage}{.48\textwidth}
\centering
\subfigure[\label{fig:nodes:riqs20}RIQS20 (proposed method), degree $75$, $\text{\#Nodes}=1932$]
  {\includegraphics[width=.8\textwidth]{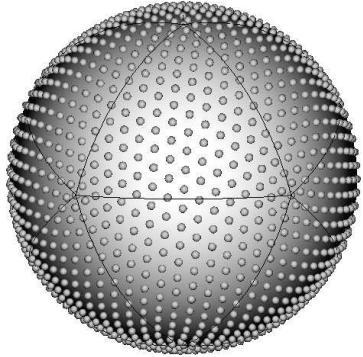}}
\end{minipage}
\hfill
\begin{minipage}{.48\textwidth}
\centering
\subfigure[\label{fig:nodes:tt}Repeating trapezoidal rules \eqref{eq:tt}, $M_\theta=30, M_\phi = 60$, $\text{\#Nodes}=1742$]
  {\includegraphics[width=.8\textwidth]{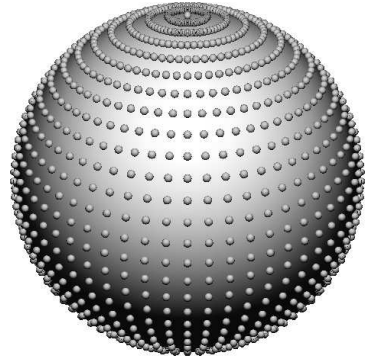}}
\end{minipage}
\caption{\label{fig:nodes}Quadrature points on $S^2$}
\end{figure}

Figure~\ref{fig:error} shows the numerical errors $|Q_K(f) - Q(f)|$ for
\[
f(\xi) = \dfrac{1}{4\pi}\dfrac{1-g^2}{(1-2g\:\xi\cdot\xi'+g^2)^{3/2}}, \quad
\xi'=\left(\dfrac19, \dfrac49, \dfrac89\right), \: g=\dfrac12,
\]
which corresponds to $p(\xi,\xi')$ in 3D RTE as a Henyey-Greenstein kernel
and satisfies $Q(f) = 1$.
In addition to RIQS20 ($+$ signs) and
repeating trapezoidal rules \eqref{eq:tt} (TT; $\times$ signs),
the Gauss-Legendre rule (GLT; $*$ signs)
to the $\theta$-direction in \eqref{eq:tt} is also examined.
The method GLT is expected to be accurate since both the Gauss-Legendre rule and
the trapezoidal rule for a periodic function are accurate
although it has also the defect in localization of quadrature points.
RIQS20 is most accurate among three quadrature rules.
\begin{figure}[t]
\centering
\includegraphics[width=.8\textwidth]{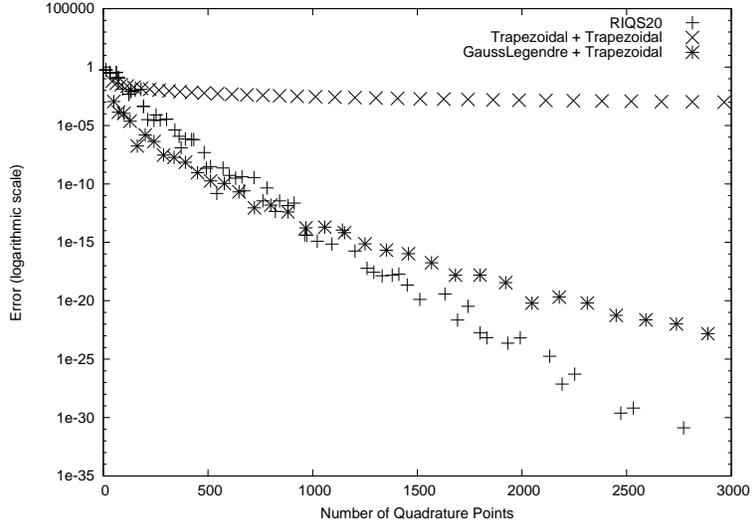}
\caption{\label{fig:error}Numerical errors}
\end{figure}

Table~\ref{tbl:weights} shows the maximum and minimum of weights
in these three quadrature rule.
More precisely,
$1920$ nodes ($99.4\%$) have weights
between $4.4\times10^{-3}$ and $7.0\times10^{-3}$
in RIQS20 with the degree $75$.
This means that almost all function values on the quadrature points
contribue equivalently to numerical quadrature
and thus it is reasonable in discretization of unknown function
in an integral equation.
\begin{table}[h]
\caption{\label{tbl:weights}Maximum and minimal weights, the number of noes associated the weights}
\begin{center}
\begin{tabular}{l|ll}
\toprule
Quadrature & Maximum & Minimum \\
\midrule
RIQS20, degree $75$         & $6.9938\times10^{-3}$ ($60$ nodes) & $2.5423\times10^{-3}$ ($12$ nodes) \\
TT, $M_\theta=M_\phi/2=30$  & $1.0966\times10^{-2}$ ($60$ nodes) & $1.1462\times10^{-3}$ ($120$ nodes) \\
GLT, $M_\theta=M_\phi/2=30$ & $1.0771\times10^{-2}$ ($120$ nodes) & $8.3443\times10^{-4}$ ($120$ nodes) \\
\bottomrule
\end{tabular}
\end{center}
\end{table}

\section{Application to 3D RTE}
Finally we show efficiency of RIQS20 in numerical computation of stationary 3D RTE.
As a numerical example,
we use an MR image of an adult human head
which consists of $181\times217\times181$ voxcels ($1$mm$^3$ cubes, Figure~\ref{fig:rteresult})
and $4.1$ million special points inside the domain.
Optical parameters reported in \cite{boas2003} are adopted.
We employ the Gauss-Seidel iteration to solve the linear equation
due to its diagonal dominance~\cite{fujiwara:jascome2011}.
The iteration is stopped with $3000$ iterations,
by which the relative residual is approximately $7\times 10^{-2}$
in the maximum norm.
\begin{figure}[t]
\centering
\includegraphics[width=.6\textwidth]{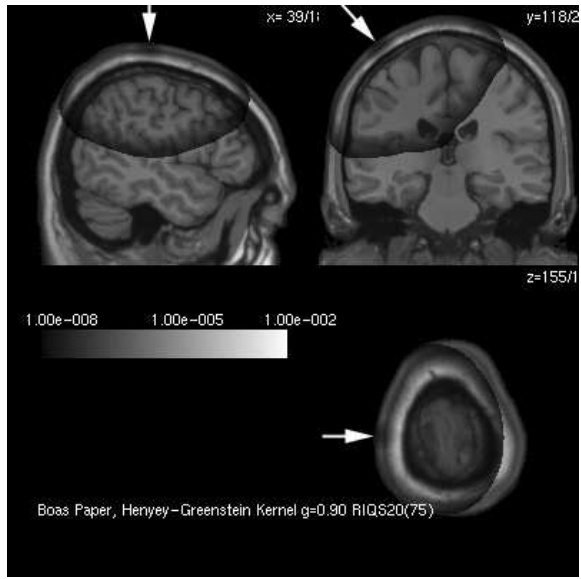}
\caption{\label{fig:rteresult}Numerical results}
\end{figure}

Using the trapezoidal rule \eqref{eq:tt} with $M_\theta=60$ and $M_\phi=120$,
the number of quadrature points
for the velocity direction $\xi \in S^2$ is $7082$
and the number of unknowns in the linear system is $28.7$ billion,
which corresponds to $214$ gigabytes in double precision.
Computational time is approximately
$87.0$ hours on Opteron 6238 (2.5GHz) with 1024 MPI processes.
On the other hand, using RIQS20 with the degree $75$,
the number of nodes on $S^2$ is $1932$
and the number of unknowns is $7.8$ billion,
which corresponds to $58$ gigabytes.
The computational time on the same environment is reduced to $6.3$ hours.
Moreover, we can process the computation on $4$ PCs (Core i7-4770, 3.4GHz)
with GPU (GeForce GTX TITAN)
and the computational time is $17.4$ hours.
The results show that RIQS20 is quite effective in computation of 3D RTE.

\begin{table}[h]
\caption{\label{tbl:comptime}Computational resources for 3D RTE}
\begin{center}
\begin{tabular}{l|c|c|c}
\toprule
Cubature Formula & \#Unknowns & Parallelization & Computational Time \\
\midrule
Trapezoidal Rule \eqref{eq:tt} & $28.7\times10^9$ & 1024 proc & $87.0$ hours \\
$M_\theta=M_\phi/2=60$ & ($214$ GB) & & \\
\midrule
RIQS20 (proposed) & $7.8\times10^9$ & 1024 proc     & $\:\:6.3$ hours \\
degree $75$       & ($58$ GB)       & 4 PC with GPU & $17.4$ hours \\
\bottomrule
\end{tabular}
\end{center}
\end{table}

\section*{Acknowledgments}
The author would like to thank Dr.~Naoya Oishi (Kyoto University)
who kindly provided an MR image and its interpretations.
This work was partially supported by JSPS KAKENHI Grant Numbers
26400198 and 25287028.


\end{document}